\documentclass[10pt]{article}
\usepackage{graphicx}
\usepackage{tikz}
\usepackage[margin=1in]{geometry}
\usepackage{amsmath, amssymb, amsfonts, amsthm}
\newtheorem{lem}{\noindent {\bf Lemma}}[section]
\newtheorem{prop}{\noindent {\bf Proposition}}[section]

\newcounter{remark}
\newcounter{example}
\catcode`@=11 
\date{}
\newcounter{defi}
\newenvironment{defi}{
\smallskip \noindent
{\bf
  Definition \arabic{section}.\arabic{defi}.
}}{\addtocounter{defi}{1}\par}

\newcommand{\NN}{\mathbb N}

\title{\bf A metric space with transfinite asymptotic dimension $2\omega+1$}
\author{\large  Yan Wu$^\ast$\qquad Jingming Zhu$^{\ast\ast}$
\footnote{
College of Mathematics Physics and Information Engineering, Jiaxing University, Jiaxing , 314001, P.R.China.
$^\ast$ E-mail: yanwu@mail.zjxu.edu.cn $^\ast\ast$ E-mail: jingmingzhu@mail.zjxu.edu.cn }}
\date{}

\begin{document}
\maketitle
\begin{center}
\begin{minipage}{0.9\textwidth}
\noindent{\bf Abstract.}
We construct a metric space whose transfinite asymptotic dimension and complementary-finite asymptotic dimension $2\omega+1$.\\


\end{minipage}
\end{center}
\footnote{
This research was supported by
the National Natural Science Foundation of China under Grant (No.11871342,11801219,11301224,11326104,11401256,11501249)

}
\begin{section}{Introduction}\

Asymptotic dimension introduced by M.Gromov \cite{Gromov} and property A by G.Yu \cite{Yu} are fundamental concepts in coarse geometry (see \cite{Nowak}).
T. Radul defined the transfinite asymptotic dimension (trasdim) which can be viewed as a transfinite extension of the asymptotic dimension and proved that for a metric space $X$, trasdim$(X)<\infty$ if and only if $X$ has asymptotic property C.  T. Radul gave examples of metric spaces with trasdim$=\infty$ and with trasdim$=\omega$, where $\omega$ is the smallest infinite ordinal number (see \cite{Radul2010}). But whether there is a metric space $X$ with $\omega<$trasdim$(X)<\infty$ (stated as``omega conjecture"in \cite{satkiewicz} by M. Satkiewicz) is unknown until recently \cite{omega+1}. In this paper, by the technique developed in \cite{omega+1}, we construct a metric spaces $X_{\omega+k}$ with trasdim$(X_{2\omega+1})=2\omega+1$, which generalized the result in \cite{omega+1} and \cite{omega+k}.

The paper is organized as follows: In Section 2, we recall some definitions and properties of transfinite asymptotic dimension. In
Section 3, we introduce a concrete metric space $X_{2\omega+1}$ whose transfinite asymptotic dimension and complementary-finite asymptotic dimension are both $2\omega+1$.
\end{section}

\begin{section}{Preliminaries}\

Our terminology concerning the asymptotic dimension follows from \cite{Bell2011} and for undefined terminology we refer to \cite{Radul2010} and \cite{omega+k}. Let~$(X, d)$ be a metric space and $U,V\subseteq X$, let
\[
\text{diam}~ U=\text{sup}\{d(x,y)| x,y\in U\}
\text{   and   }
d(U,V)=\text{inf}\{d(x,y)| x\in U,y\in V\}.
\]
Let $R>0$ and $\mathcal{U}$ be a family of subsets of $X$. $\mathcal{U}$ is said to be \emph{$R$-bounded} if
\[
\text{diam}~\mathcal{U}\doteq\text{sup}\{\text{diam}~ U~|~ U\in \mathcal{U}\}\leq R.
\]
In this case, $\mathcal{U}$ is said to be \emph{uniformly bounded}.
Let $r>0$, a family $\mathcal{U}$ is said to be\emph{ $r$-disjoint} if
\[
d(U,V)\geq r~\text{for every}~ U,V\in \mathcal{U}\text{~with~}U\neq V.
\]

In this paper, we denote
$\bigcup\{U~|~U\in\mathcal{U}\}$ by $\bigcup\mathcal{U}$, denote $\{U~|~U\in\mathcal{U}_{1}\text{~or~}~U\in\mathcal{U}_{2}\}$
by $\mathcal{U}_{1}\cup\mathcal{U}_{2}$ and
denote $\{N_{\delta}(U)~|~U\in\mathcal{U}\}$ by $N_{\delta}(\mathcal{U})$ for some $\delta>0$.
Let $A$ be a subset of $X$, we denote $\{x\in X|d(x,A)<\epsilon\}$ by $N_{\epsilon}(A)$
and denote $\{x\in X|d(x,A)\leq\epsilon\}$ by $\overline{N_{\epsilon}(A)}$ for some $\epsilon>0$.

\begin{defi}
A metric space $X$ is said to have \emph{finite asymptotic dimension} if
there exists $n\in\NN$, such that for every $r>0$,
there exists a sequence of uniformly bounded families
$\{\mathcal{U}_{i}\}_{i=0}^{n}$ of subsets of $X$
such that the family
$\bigcup_{i=0}^{n}\mathcal{U}_{i}$ covers $X$ and each $\mathcal{U}_{i}$
is $r$-disjoint for $i=0,1,\cdots,n$. In this case, we say that
the asymptotic dimension of $X$ less than or equal to $n$, which is denoted by
 asdim$X\leq n$.

We say that asdim$X= n$ if  asdim$X\leq n$ and asdim$X\leq n-1$ is not true.

\end{defi}

%
T. Radul generalized asymptotic dimension of a metric space $X$ to transfinite asymptotic dimension which is denoted by trasdim$(X)$ (see \cite{Radul2010}).
Let $Fin\mathbb{N}$ denote the collection of all finite, nonempty
subsets of $\mathbb{N}$ and let $ M \subseteq Fin\mathbb{N}$. For $\sigma\in \{\varnothing\}\bigcup Fin\mathbb{N}$, let
$$M^{\sigma} = \{\tau\in Fin\mathbb{N} ~|~ \tau \cup \sigma \in M \text{ and } \tau \cap \sigma = \varnothing\}.$$

Let $M^a$ abbreviate $M^{\{a\}}$ for $a \in \NN$. Define the ordinal number Ord$M$ inductively as follows:
\begin{eqnarray*}
\text{Ord}M = 0 &\Leftrightarrow& M = \varnothing,\\
\text{Ord}M \leq \alpha &\Leftrightarrow& \forall~ a\in \mathbb{N}, ~\text{Ord}M^a < \alpha,\\
\text{Ord}M = \alpha &\Leftrightarrow& \text{Ord}M \leq \alpha \text{ and } \text{Ord}M < \alpha \text{ is not true},\\
\text{Ord}M = \infty &\Leftrightarrow& \text{Ord}M \leq\alpha \text{ is not true for every ordinal number } \alpha.
\end{eqnarray*}

%

Given a metric space $X$, define the following collection:
\[
\begin{split}
A(X) = \{\sigma \in Fin\mathbb{N}~ |~&\text{ there are no uniformly bounded families } \mathcal{U}_i  \text{ for } i \in \sigma
 \\& \text{ such that each } \mathcal{U}_i
\text{ is } i\text{-disjoint and }\bigcup_{i\in\sigma}\mathcal{U}_i \text{~covers~} X\}.
\end{split}\]

The \emph{transfinite asymptotic dimension} of $X$ is defined as trasdim$X$=Ord$A(X)$.
\end{section}

\begin{section}{A metric space whose transfinite asymptotic dimension and complementary-finite asymptotic dimension are both $2\omega+1$}\

Let
$$X((p_1,...,p_n),(q_1,...,q_n))=\{(x_i)_{i=1}^{\sum_{j=1}^{n}q_j}\in (2^{p_1}\mathbb{Z})^{\sum_{j=1}^{n}q_j}:|\{i:x_i\notin 2^{p_{k}}\mathbb{Z}\}|\leq \sum_{j=1}^{k}q_j\text{ for }1\leq k\leq n\}$$
in which $(p_1,...,p_n),(q_1,...,q_n)\in (\mathbb{N}\cup\{0\})^n$ and $p_1\leq...\leq p_n$.\

For simplicity, we abuse the notation a little by denoting
$$X((p_1,...,p_n),(q_1,...,q_n))=\emptyset\text{ when $q_j\in \mathbb{Z}^-$ for some $j$ or when $p_1\leq...\leq p_n$ does not hold}.$$

By this new notation, the space $Y_{\omega+k}^{(n)}$ constructed in \cite{omega+k} equals to $X((k,n+k),(k,n))$ when $n\geq k$. Since trasdim$(as\bigsqcup_{n=k}^{\infty}Y_{\omega+k}^{(n)})$=trasdim$(as\bigsqcup_{n=1}^{\infty}Y_{\omega+k}^{(n)})=$trasdim$(Y_{\omega+k})=\omega+k$ for any $k\in \mathbb{N}$, so
$$\text{trasdim}(as\bigsqcup_{k=1}^{\infty}(as\bigsqcup_{n=1}^{\infty}X((k,n+k),(k,n))))\geq 2\omega.$$

By definition, $as\bigsqcup_{k=1}^{\infty}(as\bigsqcup_{n=1}^{\infty}X((k,n+k),(k,n)))\subset Y_{2\omega}$ and trasdim$(Y_{2\omega})=2\omega$, so
$$\text{trasdim}(as\bigsqcup_{k=1}^{\infty}(as\bigsqcup_{n=1}^{\infty}X((k,n+k),(k,n))))= 2\omega.$$

Now we will prove: trasdim$(\text{as}\bigsqcup_{k=1}^{\infty}(\text{as}\bigsqcup_{n=1}^{\infty}X((0,k,n),(1,k,n-k))))=2\omega+1$.\\

\begin{prop}\rm
\label{leq2omega+1}
trasdim$(\text{as}\bigsqcup_{k=1}^{\infty}(\text{as}\bigsqcup_{n=1}^{\infty}X((0,k,n),(1,k,n-k))))\leq2\omega+1$.
\end{prop}
\begin{proof}
Since $X((0,k,n),(1,k,n-k))\subset X((0,k),(1,n))$, by the proof of Proposition 3.2 in \cite{omega+1}, for any $d>0$, there exists $d-$disjoint uniformly bounded subsets families $\mathcal{U}_0$ and $\mathcal{U}_1$ such that
\[
\text{$\bigcup(\mathcal{U}_0\cup\mathcal{U}_1)=\text{as}\bigsqcup_{k=M+1}^{\infty}(\text{as}\bigsqcup_{n=1}^{\infty}X((0,k,n),(1,k,n-k)))$ for some $M>0$. }
\]

Since
$$\text{as}\bigsqcup_{k=1}^{M}(\text{as}\bigsqcup_{n=1}^{\infty}X((0,k,n),(1,k,n-k)))\subset\text{as}\bigsqcup_{k=1}^{M}(\text{as}\bigsqcup_{n=1}^{\infty}X((0,n),(M+1,n-M)))$$
and trasdim$(\text{as}\bigsqcup_{n=1}^{\infty}X((0,n),(k+1,n-k)))\leq \omega+k+1$, so
\[
\text{trasdim}(\text{as}\bigsqcup_{k=1}^{M}(\text{as}\bigsqcup_{n=1}^{\infty}X((0,k,n),(1,k,n-k))))\leq\omega+M+1<2\omega.
\]
Therefore
\[
\text{trasdim$(\text{as}\bigsqcup_{k=1}^{\infty}(\text{as}\bigsqcup_{n=1}^{\infty}X((0,k,n),(1,k,n-k))))\leq2\omega+1$ by Lemma 3.4 in \cite{omega+k}.}
\]
\end{proof}

\begin{defi}(\cite{Engelking})
Let $X$ be a metric space and let $A,B$ be a pair of disjoint subsets of $X$. We say that a subset $L\subset X$ is a \emph{partition}
of $X$ between $A$ and $B$, if there exist open sets $U,W\subset X$ satisfying the following conditions
$$A\subset U, B\subset W\text{ and }X=U\sqcup L\sqcup W.$$
\end{defi}

\begin{defi}(\cite{omega+k})
Let $X$ be a metric space and let $A,B$ be a pair of disjoint subsets of $X$. For any $\epsilon>0$, we say that a subset $L\subset X$ is an \emph{$\epsilon$-partition}
of $X$ between $A$ and $B$, if there exist open sets $U,W\subset X$ satisfying the following conditions
$$A\subset U, B\subset W, X=U\sqcup L\sqcup W, d(L,A)>\epsilon\text{ and }d(L,B)>\epsilon$$
Clearly, an $\epsilon$-partition $L$ of $X$ between $A$ and $B$ is a partition of $X$ between $A$ and $B$.
\end{defi}

\begin{lem}\rm(\cite{omega+k})
\label{partition2}
Let $L_0\doteq [0,B]^n$ for some $B>0$, $F_i^{+}$, $F_i^{-}$ be the pairs of opposite faces of $L_0$, where $i=1,2,\cdots,n$ and let $0<\epsilon<\frac{1}{6}B$.
For $k=1,2,\cdots, n$, let
$\mathcal{U}_k$ be an $\epsilon$-disjoint and $\frac{1}{3}B$-bounded family of subsets of $L_{k-1}$. Then there exists an $\epsilon$-partition $L_k$ of $L_{k-1}$ between $F_k^+\cap L_{k-1}$ and $F_k^-\cap L_{k-1}$ such that $L_k\subset L_{k-1}\cap(\bigcup \mathcal{U}_k)^c$.
\end{lem}

\begin{lem}\rm(see \cite{Engelking}, Lemma 1.8.19)
\label{partition}
Let $F_i^{+}$, $F_i^{-}$, where $i\in\{1,\ldots,n\}$, be the pairs of opposite faces of $I^n\doteq [0,1]^n$. If $I^n=L_0'\supset L_1'\supset \ldots\supset L_n'$ is a decreasing sequence of closed sets such that $L_i'$ is a partition of $L_{i-1}'$ between $L_{i-1}'\cap F_i^{+}$ and $L_{i-1}'\cap F_i^{-}$ for $i\in\{1,2,\ldots,n\}$, then $L_{n}'\neq \emptyset$.
\end{lem}

\begin{prop}\rm
\label{geq2omega+1}
trasdim$(\text{as}\bigsqcup_{k=1}^{\infty}(\text{as}\bigsqcup_{n=1}^{\infty}X((0,k,n),(1,k,n-k))))>2\omega$.
\end{prop}
\begin{proof}
If not, trasdim$(\text{as}\bigsqcup_{k=1}^{\infty}(\text{as}\bigsqcup_{n=1}^{\infty}X((0,k,n),(1,k,n-k))))\leq 2\omega$. Let
$$
X=\text{as}\bigsqcup_{k=1}^{\infty}(\text{as}\bigsqcup_{n=1}^{\infty}X((0,k,n),(1,k,n-k))).
$$
Then $\forall a\in\mathbb{N}$, Ord$A(X)^a\leq \omega+m$ for some $m=m(a)\in\mathbb{N}$. By definition, $\forall \tau\in Fin\mathbb{N}$ with $a\notin \tau$ and $|\tau|=m+1$, Ord$A(X)^{\{a\}\sqcup \tau}\leq n$ for some $n=n(a,\tau)$. And then, for any $\sigma\in Fin\mathbb{N}$ with $|\sigma|=n+1$ and $(\{a\}\sqcup\tau)\cap\sigma=\emptyset$, $\{a\}\sqcup\tau\sqcup\sigma\notin A(X)$. For the chioce
$$\tau=\{b,b+1,...,b+m\},\sigma=\{c,c+1,...,c+n\}\text{ for }c>b+m,b>a,$$
we have $a-$disjoint $B-$bounded subset family $\mathcal{U}$ and $b-$disjoint $B-$bounded subset families $\mathcal{V}_1,...,\mathcal{V}_{m+1}$ and $c-$disjoint $B-$bounded subset families $\mathcal{W}_1,...,\mathcal{W}_{n+1}$ such that $\mathcal{U}\cup(\bigcup_{i=1}^{m+1}\mathcal{V}_i)\cup(\bigcup_{j=1}^{n+1}\mathcal{W}_j)$ covers $X$.
So $\mathcal{U}\cup(\bigcup_{i=01}^{m+1}\mathcal{V}_i)\cup(\bigcup_{j=1}^{n+1}\mathcal{W}_j)$ covers $X((0,m+1,m+n+2)(1,m+1,n+1))\cap[0,6B]^{m+n+3}$. Without lose of generality, we can assume that $1\ll a\ll m\ll b\ll n\ll c\ll B$.\

We assume that $p=\frac{6B}{2^{m+n+2}}\in \mathbb{N}$.
Taking a bijection $\psi:\{1,2,\cdots,p^{m+n+3}\} \to \{0,1,2,\cdots,p-1\}^{m+n+3}$, let
$$Q(t)=\prod_{j=1}^{m+n+3}[2^{m+n+2}\psi(t)_{j},2^{m+n+2}(\psi(t)_{j}+1)],\text{in which }\psi(t)_{j}\text{ is the $j$th coordinate of }\psi(t).$$
Let $\mathcal{Q}=\{Q(t)~|~t\in\{1,2,\cdots,p^{m+n+3}\}\}$, then $[0,6B]^{m+n+3}=\bigcup_{Q\in\mathcal{Q}}Q$.

Let $L_0=[0,6B]^{m+n+2}$. By Lemma \ref{partition2},
since $N_{2^{m+n+2}}(\mathcal{W}_1)$ is $(c-2^{m+n+3})$-disjoint and $(2^{m+n+3}+B)$-bounded, there exists a $(c-2^{m+n+3})$-partition $L_1$ of $[0,6B]^{m+n+3}$ such that $L_1\subset (\bigcup N_{2^{m+n+2}}(\mathcal{W}_1))^c\cap [0,6B]^{m+n+3}$ and $d(L_1,F_1^{+/-})>c-2^{m+n+3}$.\

Let $\mathcal{M}_1=\{Q\in \mathcal{Q}|Q\cap L_1\neq \emptyset\}$ and $M_1=\bigcup \mathcal{M}_1$. Since $L_1$ is a $(c-2^{m+n+3})$-partition of $[0,6B]^{m+n+3}$ between $F_1^+$ and $F_1^-$, then  $M_1$ is a $(c-2^{m+n+4})-$partition of $[0,6B]^{m+n+3}$ between $F_1^+$ and $F_1^-$. i.e. $[0,6B]^{m+n+3}=M_1\sqcup A'_1 \sqcup B'_1$ such that $A'_1$, $B'_1$ are open in $[0,6B]^{m+n+3}$ and $A'_1$, $B'_1$ contain two opposite facets $F_1^-$, $F_1^+$ respectively.
Let $L'_1=\partial_{m+n+2}M_1\doteq\bigcup\{\partial_{m+n+2}Q|Q\in \mathcal{M}_1\}$, then $[0,6B]^{m+n+3}\setminus (L'_1\sqcup A'_1 \sqcup B'_1)$ is the union of some disjoint open $(m+n+3)$-dimensional cubes with length of edge $= 2^{m+n+2}$.
So $L'_1$ is a $(c-2^{m+n+4})-$partition of $[0,6B]^{m+n+3}$  between $F_1^+$ and $F_1^-$ and $L'_1\subset (\bigcup\mathcal{W}_1)^c\cap [0,6B]^{m+n+3}$.

For $N_{2^{m+n+2}}(\mathcal{W}_2)$, by Lemma \ref{partition2} and similar argument above, there exists a $(c-2^{n+m+3})$-partition $L_2$ of $L'_1$ such that $L_2\subset (\bigcup  N_{2^{m+n+2}}(\mathcal{W}_2))^c\cap L'_1$ and $d(L_2,F_2^{+/-})>c-2^{n+m+3}$.\

Let $\mathcal{M}_2=\{Q\in \mathcal{M}_1~|~Q\cap L_2\neq \emptyset\}$ and $M_2=\bigcup \mathcal{M}_2$. Since $L_2$ is a $(c-2^{n+m+3})$-partition of $L'_1$ between $L'_1\cap F_2^+$ and $L'_1\cap F_2^-$, then $M_2 \cap L_1'$ is  a $(c-2^{n+m+4})-$partition of $L'_1$  between $L'_1\cap F_2^+$ and $L'_1\cap F_2^-$ i.e. $L'_1=(M_2\cap L_1')\sqcup A'_2 \sqcup B'_2$ such that $A'_2$, $B'_2$ are open in $L'_1$ and $A'_2$, $B'_2$ contain two opposite facets $L'_1\cap F_2^-$, $L'_1\cap F_2^+$ respectively.
Let $L'_2=\partial_{m+n+1}M_2\doteq\bigcup\{\partial_{m+n+1}Q|Q\in \mathcal{M}_2\}$, then $L'_1\setminus (L'_2\sqcup A'_2 \sqcup B'_2)$ is the union of some disjoint open $(m+n+2)$-dimensional cubes with length of edge $= 2^{m+n+2}$.
So $L'_2$ is also a $(c-2^{n+m+4})-$partition of $L'_1$ between $L'_1\cap F_2^+$ and $L'_1\cap F_2^-$ and $L'_2\subset (\bigcup (\mathcal{W}_1 \cup \mathcal{W}_2))^c\cap [0,6B]^{m+n+3}$.

After $n+1$ steps above, we have $L'_{n+1}$ to be a $(c-2^{n+m+4})-$partition of $L'_{n}$ and $$L'_{n+1}\subset (\bigcup (\mathcal{W}_1 \cup...\cup \mathcal{W}_{n+1}))^c\cap [0,6B]^{m+n+3}$$
Note that $L'_{n+1}\subset \{(x_i)_{i=1}^{n+m+3}\in \mathbb{R}^{n+m+3}||\{j|x_j\notin 2^{m+n+2}\mathbb{Z}\}|\leq m+2\}$. For every point $(x_i)_{i=1}^{n+m+3}\in \{(x_i)_{i=1}^{n+m+3}\in \mathbb{R}^{n+m+3}||\{j|x_j\notin 2^{m+n+2}\mathbb{Z}\}|\leq m+2\}$, there exists $i(1),...,i(m+2)\in\mathbb{N}$ such that
$$x_{i(j)}\in[n_{i(j)}2^{m+2},(n_{i(j)}+1)2^{m+2}]\text{ for }j\in\{i(1),...,i(m+2)\} \text{ for some }n_{i(j)}\in \mathbb{Z}$$
$$x_{i(j)}\in2^{n+m+3}\mathbb{Z}\text{ for }j\notin\{i(1),...,i(m+2)\}.$$

Let
\begin{multline*}
  \mathcal{Q}'\doteq \{\prod_{i=1}^{n+m+3}I_i| I_{i(j)}=[n_{i(j)}2^{m+2},(n_{i(j)}+1)2^{m+2}]\text{ for }j\in\{i(1),...,i(m+2)\} \text{ for some }n_{i(j)}\in \mathbb{Z}\text{ and}\\
  I_{i(j)}=\{x_{i(j)}\}\subset2^{n+m+3}\mathbb{Z}\text{ for }j\notin\{i(1),...,i(m+2)\} \text{ and some }\{i(1),...,i(m+2)\}\subset \mathbb{N} \}.
\end{multline*}

So $\mathcal{Q}'$ is a subset family in which each $\prod_{i=1}^{n+m+3}I_i$ is an $(m+2)$-dimensional cube with length of edge$=2^{m+1}$. Let $\mathcal{P}=\{Q'\in \mathcal{Q}'~| Q'\subset L'_{n+1}\}$.

By Lemma \ref{partition2},
since $N_{2^{m+1}}(\mathcal{V}_1)$ is $(b-2^{m+2})$-disjoint and $(2^{m+2}+B)$-bounded, there exists a $(b-2^{m+2})$-partition $L_{n+2}$ of $L'_{n+1}$ such that $L_{n+2}\subset (\bigcup N_{2^{m+1}}(\mathcal{V}_1))^c\cap L'_{n+1}$ and $d(L_{n+2},F_{n+2}^{+/-})>(b-2^{m+2})$.\

Let $\mathcal{M}_{n+2}=\{Q'\in \mathcal{P}|Q'\cap L_{n+2}\neq \emptyset\}$ and $M_{n+2}=\bigcup \mathcal{M}_{n+2}$. Since $L_{n+2}$ is a $(b-2^{m+2})$-partition of $L'_{n+1}$ between $F_{n+2}^+$ and $F_{n+2}^-$, then  $M_{n+2}$ is a $(b-2^{m+3})-$partition of $L'_{n+1}$ between $L'_{n+1}\cap F_{n+2}^+$ and $L'_{n+1}\cap F_{n+2}^-$. i.e. $L'_{n+1}=M_1\sqcup A'_{n+2} \sqcup B'_{n+2}$ such that $A'_{n+2}$, $B'_{n+2}$ are open in $L'_{n+1}$ and $A'_{n+2}$, $B'_{n+2}$ contain two opposite facets $L'_{n+1}\cap F_{n+2}^+$, $L'_{n+1}\cap F_{n+2}^-$ respectively.
Let $L'_{n+2}=\partial_{m+1}M_{n+2}\doteq\bigcup\{\partial_{m+1}Q'|Q'\in \mathcal{M}_{n+2}\}$, then $L'_{n+1}\setminus (L'_{n+2}\sqcup A'_{n+2} \sqcup B'_{n+2})$ is the union of some disjoint open $(m+2)$-dimensional cubes with length of edge $= 2^{m+1}$.
So $L'_{n+2}$ is a $(b-2^{m+3})-$partition of $L'_{n+1}$  between $L'_{n+1}\cap F_{n+2}^+$ and $L'_{n+1}\cap F_{n+2}^-$ and $L'_{n+2}\subset (\mathcal{V}_1)^c\cap L'_{n+1}$.

After $m+1$ steps above, we have $L'_{m+n+2}$ to be a $(b-2^{m+3})-$partition of $L'_{m+n+1}$ which contains a connected component $C$ intersecting two opposite facets $F_{m+n+3}^+,F_{m+n+3}^-$ and
$$L'_{n+m+2}\subset (\bigcup (\mathcal{V}_1 \cup...\cup \mathcal{V}_{m+1}))^c\cap L'_{n+1}.$$
Note that $L'_{n+m+2}\subset \{(x_i)_{i=1}^{n+m+3}\in \mathbb{R}^{n+m+3}||\{j|x_j\notin 2^{m+n+2}\mathbb{Z}\}|\leq m+2\text{ and }|\{j|x_j\notin 2^{m+1}\mathbb{Z}\}|\leq 1\}$. So $$C\subset (\bigcup\bigcup_{i=1}^{m+1} \mathcal{V}_{i})^c\cap(\bigcup\bigcup_{j=1}^{n+1} \mathcal{W}_j )^c \cap X((0,m+1,m+n+2)(1,m+1,n+1))\cap[0,6B]^{m+n+3}.$$

Since $C$ is connected and intersects with two opposite facets $F_{m+n+3}^+,F_{m+n+3}^-$, then $C$ can not be covered by $\mathcal{U}$, which implies $\mathcal{U}\cup(\bigcup_{i=1}^{m+1}\mathcal{V}_i)\cup(\bigcup_{j=1}^{n+1}\mathcal{W}_j)$ can not cover $X$ and hence leads to a contradiction.
\end{proof}

\end{section}

{\bf Acknowledgments.} The author wish to thank the reviewers for careful
reading and valuable comments. This work was supported by NSFC grant of P.R. China (No.11871342,11801219,11301224,11326104,11401256,\\
11501249). And the authors want to thank V.M. Manuilov for helpful discussion.

\providecommand{\bysame}{\leavevmode\hbox to3em{\hrulefill}\thinspace}
\providecommand{\MR}{\relax\ifhmode\unskip\space\fi MR }
\providecommand{\MRhref}[2]{%
  \href{http://www.ams.org/mathscinet-getitem?mr=#1}{#2}
}
\providecommand{\href}[2]{#2}


\begin{thebibliography}{10}

\bibitem{Gromov}
M.~Gromov, \emph{Asymptotic invariants of infinite groups.} in: Geometric Group Theory, Vol.2, Sussex, 1991, in: Lond. Math. Soc. Lect. Note Ser., vol.182, Cambridge Univ. Press, Cambridge, (1993), 1--295.

\bibitem{Yu}
G.~Yu, \emph{The coarse Baum-Connes conjecture for spaces which admit a uniform embedding into Hilbert space.} Invent. Math. 139 (2000), 201--204.

\bibitem{Nowak}
P.~Nowak, G.~Yu, \emph{Large Scale Geometry.} EMS Textbk. Math., European Mathematical Society, Z$\ddot{u}$rich, 2012.











\bibitem{Radul2010}
T.~Radul, \emph{On transfinite extension of asymptotic dimension.} Topol. Appl. 157 (2010), 2292--2296.


\bibitem{satkiewicz}
M.~Satkiewicz, \emph{Transfinite Asymptotic Dimension.} arXiv:1310.1258v1,
2013.


\bibitem{omega+1}
Jingming Zhu, Yan Wu, \emph{A metric space with its transfinite asymptotic
dimension $\omega+1$.} arxiv.org:1908.00434

\bibitem{omega+k}
Jingming Zhu, Yan Wu, \emph{Examples of metric spaces with asymptotic property C.} arxiv.org:1912.02103


\bibitem{Bell2011}
G.~Bell, A. Dranishnikov, \emph{Asymptotic dimension in Bedlewo.} Topol. Proc. 38 (2011), 209--236.











\bibitem{Engelking}
R.~Engelking, \emph{Theory of Dimensions: Finite and Infinite.} Heldermann Verlag, 1995.


%



























\end{thebibliography}
\end{document}